\title{Flow Cytometry Based State Aggregation\\ of a Stochastic Model of Protein Expression}
\author{Anahita Mirtabatabaei, Francesco Bullo, Mustafa Khammash}
\documentclass[letterpaper,10pt]{IEEEtran}
\IEEEoverridecommandlockouts
\makeatletter
\let\IEEEproof\proof
\let\IEEEendproof\endproof
\let\proof\@undefined
\let\endproof\@undefined
\makeatother

\usepackage{amsmath,amsthm,amscd,amssymb} 
\let\proof\IEEEproof
\let\endproof\IEEEendproof

\usepackage{color,url,graphicx,epsfig,amsfonts,subfigure}
\usepackage{cancel,enumerate,version}
\newtheorem{theorem}{Theorem}[section]

\newtheorem{proposition}[theorem]{Proposition}

\theoremstyle{definition}

\theoremstyle{remark}
\newtheorem{remark}{Remark}

\newcommand{\ml}{\left[ \begin{array}{ccccccccc}}
\newcommand{\mr}{\end{array} \right]}

\begin{document}

\maketitle

\begin{abstract}
In this article, we introduce the new approach \emph{fluorescence grid based aggregation (FGBA)} to justify a dynamical model of protein expression using experimental fluorescence histograms. In this approach, first, we describe the dynamics of the gene-protein system by a chemical master equation (CME), while the protein production rates are unknown. Then, we aggregate the states of the CME into unknown group sizes.  We show that these unknown values can be replaced by the data from the experimental fluorescence histograms. Consequently, final probability distributions correspond to the experimental fluorescence histograms. 



\end{abstract}
\section{Introduction}

In the study of protein expression, \emph{flow cytometry} is a promising technique for the analysis of protein regulatory system \cite{HL-AO:07,FI-JH-CC-JC:03,BM-BT-MK:09}. 
In a cell colony, flow cytometery measures single cell's fluorescent intensity, which represents the protein concentration, and draws a \emph{fluorescence histogram}. A fluorescence histogram of a cell colony is a plot of the cell count versus measured fluorescent intensity \cite{JVW:91}. 
In theory, the process of protein expression has been stochastically analyzed to generate a probability distribution of protein concentration.
However, there are two deficiencies to this analysis. First, generated probability distributions do not represent the experimental fluorescence histograms, since the relation between fluorescent intensity and protein concentration is unknown. Second, the protein production rate,  which is a key parameter in stochastic analysis of expression, is not known for different expression states of a gene.

In this paper, we study the expression of a protein called \emph{Ag43} by a gene named \emph{agn43}.
This protein is not involved in feedback regulation, and instead the encoding gene uses a mechanism of generating multiple \emph{phases} in order to regulate the protein production. 
Phase variation describes changes in the \emph{expression state} of the gene that results
in mixed cell cultures in a colony \cite{MW-IRH:08}. 
 A gene is called to have an \emph{On}, \emph{Partial}, or \emph{Off} expression state, if it produces protein with a high, low, approximately zero rate, respectively. 
 In the mechanism of agn43 regulation, between phases with On and Off expression states, the gene enters intermediate phases that act as buffers and prevent back and forth switching. 
Recently, Lim et al. (2007) proposed a dynamical model for the phase variation of agn43 and identified a third expression state, Partial, for the gene.
They verified the model deterministically, and computed the phase variation rates of the gene. 
 However, the protein production rates in those three expression states are unknown,  and the dynamics of the protein production is not analyzed.

As our main contribution, we introduce a new approach to justify the dynamical model of gene-protein system by the experimental fluorescence histograms. We call this approach the \emph{fluorescence grid based aggregation (FGBA)}. First, we compute the rate of increase in cell's fluorescent intensity by the steady state histograms. This rate has a linear relation with the protein production rate. Second, assuming that the stochastic dynamics of the gene-protein system is a Markov process,
 we describe this system  by a chemical master equation (CME), while the protein production rates, for different expression states, are unknown. 
Third, we aggregate the states of the CME into groups with unknown sizes, and compute the dynamics of the aggregated system. 
 Aggregation of Markov chains, also known as sparse grid approximation \cite{MH-CB-LS-SM-HB:07} and projection through interpolation \cite{BM-MK:08}, has been employed to the gene regulatory networks in order to reduce the computation time. However, in those studies, the number of states being aggregated and the protein production rates were known, as opposed to our method.
In FGBA method, we aggregate the CME based on the fluorescence grid sizes in experimental fluorescence histograms. 
By employing this method on the CME~\eqref{pdot}, we achieve the following goals: (1) we eliminate the dependence of the CME on protein number, and hence, its dependence on unknown protein production rates; (2) we define CME as a function of fluorescent intensity, 
solving which gives final probability distributions that correspond to the experimental fluorescence histograms; 
and (3) we reduce the size of the differential CME to reduce the computation time. 
Finally, we find an upper bound for the evolution of the error caused by employing FGBA method.


The paper develops as follows. The remainder of this section reviews the studied gene and protein. The deterministic and stochastic analysis of the gene-protein system is discussed in Sections~\ref{deterministic} and \ref{stoch}, respectively. 
The FGBA method and its error are presented in Subsection~\ref{FGBAsec}. Numerical results are provided in Section~\ref{numericalresults}. Finally, some conclusions are drawn in Section~\ref{conclusion}.


\subsection{Gene-Protein System}\label{biology}

\emph{Antigen 43 (Ag43)} is an outer membrane protein in the bacteria \emph{Escherichia coli} and is described as its "most abundant phase varying outer membrane protein'' \cite{IRH-PO:99}. This protein is encoded by a single gene called \emph{agn43} or \emph{flu}. Flu is an abbreviation of fluffing due to the fact that the production of Ag43 causes interspecies cell aggregation by Ag43-Ag43 interaction. Hence, the expression of this protein enhances biofilm formation. As mentioned above, the phase variation of agn43 regulates production of Ag43.
In agn43, phase variation is performed by an \emph{epigenetic switch}. An epigenetic switch can be defined as a heritable yet reversible switch in gene expression state, which is not mediated by a change in DNA sequence \cite{MW-IRH:08}. Therefore, agn43 is a controllable toggle switch. As a practical device, a toggle switch forms a synthetic, addressable cellular memory unit and has implications for biotechnology, biocomputing and gene therapy \cite{TSG-CRC-JJC:00}.

The dynamics of phase variation in agn43 is studied separately by \cite{HL-AO:07} and \cite{MW-IRH:08}. A schematic of the model proposed by \cite{HL-AO:07} is illustrated in Figure~\ref{transcriptionfig}. 
The methylation state of three GATC sequences along the gene  
 decides whether the expression is On (methylated) or Off (unmethylated).
The methylation state of the GATC sites is determined by competitive binding between OxyR, a global oxidative stress protein, and DNA adenine methylase (Dam). 
Since there is no DNA demethylation reaction, gene replication is essential to the phase variation. 
After each replication: fully methylated agn43 ($M_F$), whose expression state is On, becomes hemimethylated ($M_H$); the hemimethylated agn43 generates one hemimethylated and one unmethylated and naked agn43 ($U_N$); and the gene in the rest of phases 
keeps its initial phase. 
In Lim's model, the expression state of $M_H$ is said to be either On or Partial, while we assume this expression to be On, according to the heritable expression state of agn43 \cite{MW:06}. 
 OxyR can bind to  $U_N$ and generate an unmethylated agn43 with OxyR ($U_O$). In Lim's phase variation model, agn43 in $U_N$ and $U_O$ phases partially transcribes protein, as opposed to the model proposed by Marjan et al. (2008). In Lim's model the DNA in $U_O$ phase can undergo a conformational change, giving rise to an Off phase ($O$) with Off expression state.


\section{Deterministic Analysis}\label{deterministic}
The deterministic dynamics of the agn43-Ag43 system can be divided into three parts: gene's phase variation, protein production, and the gene replication during cell division.

\subsection{Dynamics of Phase Variation}\label{phaserate}
We briefly review the dynamics of agn43 phase variation in Lim's model. According to Section~\ref{biology}, five phases and three expression states are assigned to agn43: $M_F$, $M_H$, $U_N$, $U_O$, and $O$ phases with On, On, Partial, Partial, and Off expression states, respectively.
As illustrated in Figure~\ref{transcriptionfig}, the dynamics of these five phases can be written as:
\begin{gather*}
\dot{M}_F(t) = k_M M_H(t), \hspace{.1in} \dot{M_H}(t)  = k_H U_N(t) - k_M M_H(t),\\
\dot{U}_N(t) = k_{-O}U_O(t) - (k_O+k_H)U_N(t), \\
\dot{U}_O(t)  = k_{-R}O(t) + k_O U_N(t) -(k_{-O}+k_R)U_O(t), \\
\dot{O}(t)  = -k_{-R}O(t) +k_R U_O(t).
\end{gather*}
According to the supplementary methods of \cite{HL-AO:07}, $k_M=4.3$, $\frac{k_O}{k_{-O}}=3.7$, and $\frac{k_R}{k_{-R}}=15.8$. 
Based on our sensitivity analysis, we used $k_H=0.4$.
Here, we need two more equalities to compute all the phase varying rates. Owing to rare On-Off switching of the agn43 ($7\times 10^{-3}$ cells per generation), $k_H \ll k_O$, hence we assume that $k_O= 1000 k_H$. Finally, considering the steady state of the system, it can be computed that $k_R=0.118k_O.$
\begin{figure}
  \centering
    \includegraphics[width=2.5in,keepaspectratio]{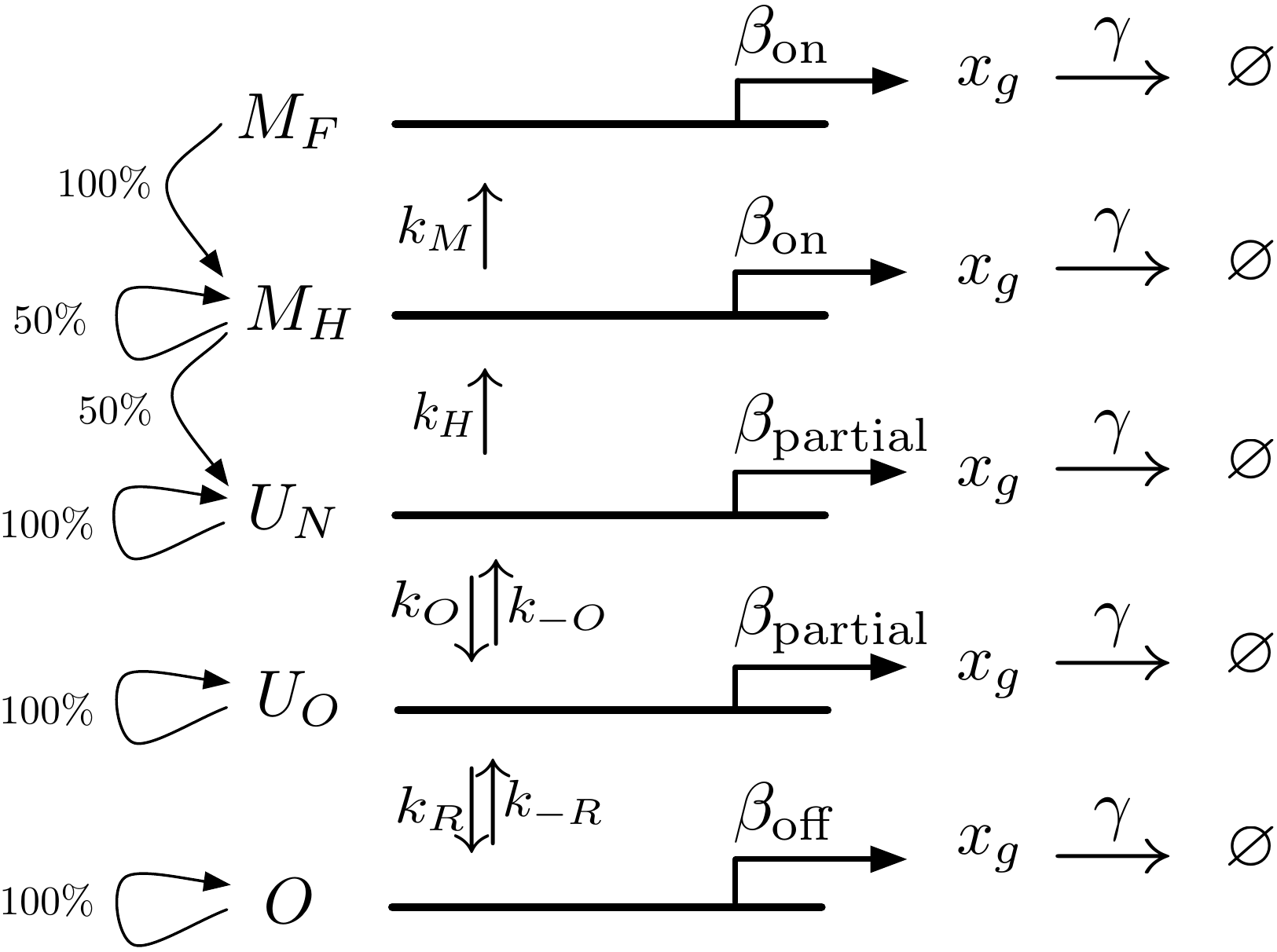}
  \caption{Lim's dynamical model of agn43-Ag43. The gene has five phases $M_F$, $M_H$, $U_N$, $U_O$, and $O$, with On, On, Partial, Partial, and Off expression states, respectively. 
Depending on the expression state, the protein $x_g$ is produced with three different rates $\beta_{\text{on}}$, $\beta_{\text{partial}}$, and $\beta_{\text{off}}$, but degrades with fixed rate $\gamma$.
The arrows on the left represent the effect of replication on the phase of the gene.}
  \label{transcriptionfig}
\end{figure}

\subsection{The Dynamics of the Protein Production}\label{transcription}

A useful function that describes protein production rate in many real genes is \emph{Hill function} \cite{UA:06}. According to this function, in the absence of activator and repressor, the protein production rate is constant. As discussed in Section~\ref{biology}, there is no feedback regulation in the production of Ag43 and the concentration of external factors, OxyR and Dam, during cell growth is constant by over expression. Thus, the dynamics of protein production can be described by
\begin{equation}\label{gfp}
\dot{x}_g(t) = \beta - \gamma x_g(t),
\end{equation}
where $x_g(t)$, $\beta$, and $\gamma$ represent the concentration, production rate, and degradation rate of the reporter protein, respectively. 
In the experiments by \cite{HL-AO:07}, \emph{green fluorescent protein} (GFP) is used as a reporter, and its production is regulated by agn43. GFP exhibits fluorescence in the cell, that can be measured by flow cytometry.
Based on the method of generating and amplifying the expression of GFP in \cite{HL-AO:07}, we assume that there is a linear relation between the rates of Ag43 production and GFP production in the cell. However, their degradation is independent of each other, and the latter is measurable by flow cytometer. Therefore, we consider the dynamics of GFP production to verify the model by experimental results.

The rate $\gamma$ is the sum of \emph{dilution} and \emph{degradation} rates. Dilution is the reduction of protein density due to increase in cell volume. Since a flow cytometer measures the total fluorescence of a cell rather than the density of fluorescence,
the dilution rate is zero here. Degradation rate is computed by protein's half life $\tau$ while its production rate is zero. That is, $x_g(\tau)= x_g(0)/2=x_g(0) e^{- \gamma \tau},$ and thus $\gamma= \ln{2}/\tau$.
Half life of wild type GFP is 26 hours \cite{PC-CT:99}, and one generation takes 85 minutes \cite{HL-AO:07}, therefore, $\gamma$ is equal to $0.0378$ protein per generation. 

The protein production rate $\beta$ depends on the expression state of the gene, On, Partial, or Off.
Consider a gene that remains in one expression state as time goes to infinity. Then, the protein concentration of the cell reaches a steady state $x_{g,\infty}$, and thus $\lim_{t \rightarrow \infty}\dot{x}_{g}(t)=0$. It follows from equation~\eqref{gfp} that $\beta = \gamma x_{g,\infty}$ protein per generation. Our tool to compute $x_{g,\infty}$ is the experimental fluorescence histogram, e.g., Figure~\ref{histogramsample}. 
However, for each expression state, such histogram gives us the fluorescent intensity of a cell in steady state in arbitrary units (a.u.) instead of the protein concentration. 

In a cell, the fluorescent intensity $x_f$  depends linearly on protein (GFP) concentration, see \cite{JA-HC-TK-IO-CC-GH-CP-DK:05} and \cite{MS-JO-WH:05}.  That is, $x_f(t) = \mu x_g(t)$, where we call $\mu$ the \emph{fluorescence-GFP ratio}, and its value unknown. 
Taking the derivative of both sides gives
\begin{equation}\label{linearf}
\dot{x}_f(t) = \mu (\beta - \gamma x_g(t)) = \mu \beta - \gamma x_f(t) = \beta_f - \gamma x_f(t),
\end{equation}
where $\beta_f$ denotes the rate of increase in fluorescent intensity of the cell.
According to Figure~\ref{histogramsample}, the steady state fluorescent intensity $x_{f,\infty}$ of a cell  whose agn43 has On, Partial, or Off expression state is $10^{3.5}$, $10^{1.8}$, or $10$ a.u., respectively.
It follows from $\beta_f = \gamma x_{f,\infty}$ that $\beta_{f,\text{on}} = 238$, $\beta_{f,\text{partial}} = 3$, and $\beta_{f,\text{off}} = 0.37$ a.u. per generation. 
\begin{figure}
  \centering
    \includegraphics[width=2.6in,keepaspectratio]{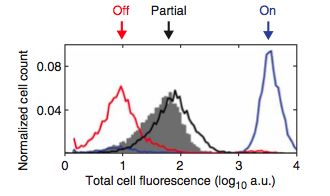}
  \caption{ Three fluorescence histograms of cell colonies in three different expression states after 20 hours. 
This plot tells us that the steady state fluorescent intensity of a cell whose agn43 has On, Partial, or Off expression state is $10^{3.5}$, $10^{1.8}$, or $10$ a.u., respectively.
 Reprinted figure with permission from \cite{HL-AO:07}. \copyright 2007, by Nature Publishing Group.} 
  \label{histogramsample}
\end{figure}

\subsection{Replication Rates}\label{replication}

Replication of the cell has two effects in our model. First, we assume that the protein concentration of the cell becomes half of its initial value. This assumption is based on two reasons: "in immunofluorescence studies of Ag43-producing E. coli, the protein is seen evenly distributed over the surface of the entire cell'' \cite{IH-MM-PO:97}; and, in our stochastic analysis we have observed that
employing binomial distribution for protein concentration after replication has a negligible effect on the final probability distribution, see Figure~\ref{discrete}.
Second, after replication the gene's phase vary: any $M_F$ gene becomes $M_H$; half of $M_H$ genes become $U_N$, and the other half remain $M_H$; and genes in the rest of phases keep their initial phase, see Figure~\ref{transcriptionfig}.

\section{Stochastic Analysis}\label{stoch}
We aim to describe the dynamics of the protein expression by the phase varying gene agn43 by a Markovian process. In other words, we compute the probability of a cell being in any \emph{configuration}, which is here determined by its gene's phase plus its protein concentration, as a function of time.
Therefore, a cell's configuration changes based on: (1) phase variation rates, (2) protein production and degradation rates, and (3) replication rates, see Figure~\ref{states}.
\begin{figure}
  \centering
    \includegraphics[width=2.4in,keepaspectratio]{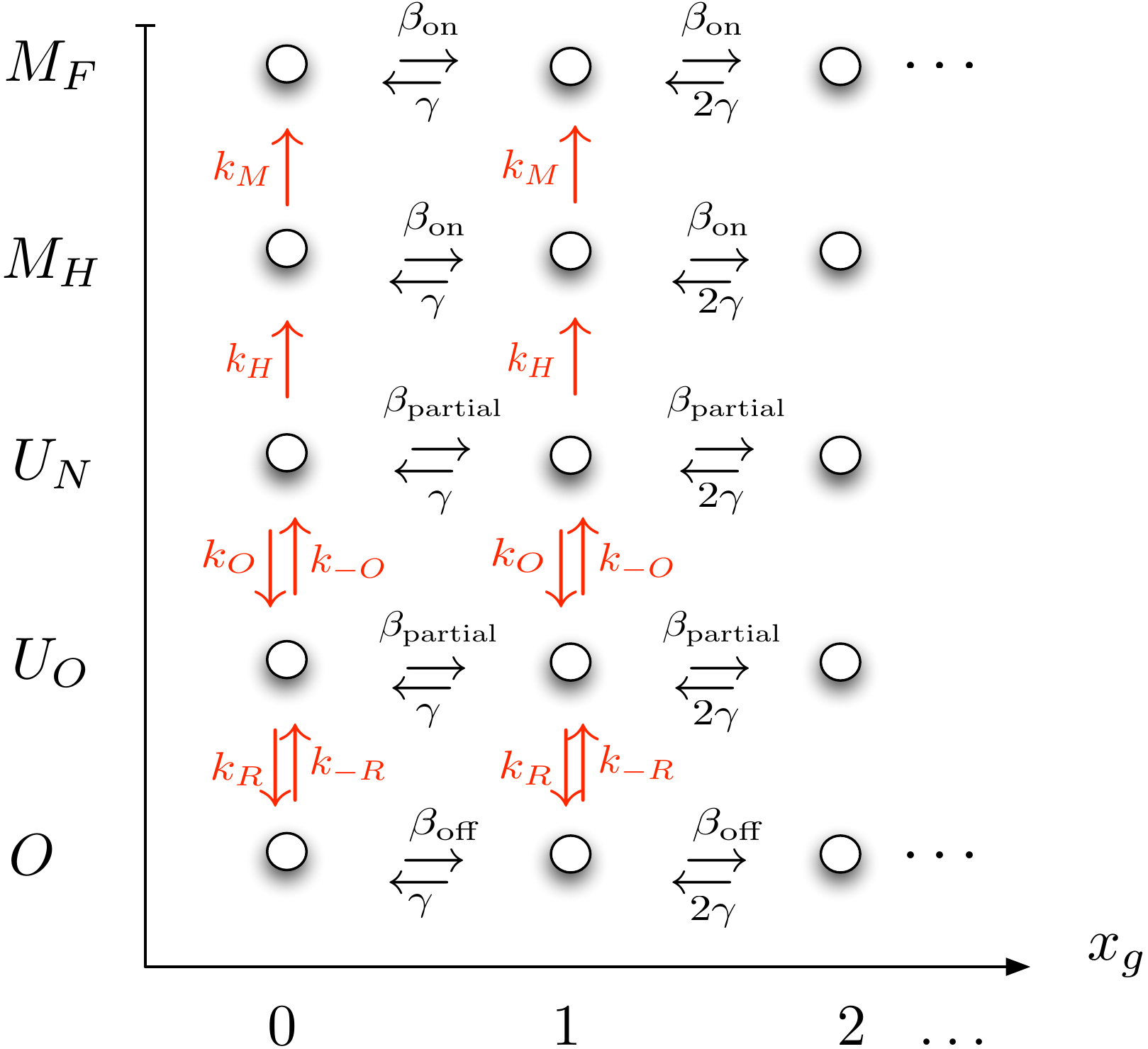}
  \caption{Each circle represents one possible configuration for a cell that contains agn43, based on the cell's protein concentration (horizontal axis) and its gene's phase (vertical axis). 
The transitions between configurations, shown by arrows, is possible through phase variation (red arrows) or change in protein concentration (black arrows). For brevity, the effect of cell replication on protein concentration is not illustrated.}
  \label{states}
\end{figure} 
For each cell, the probability of having any such configuration is a function of time, and the union of those probabilities makes up the probability distribution vector $P(t)$. More specifically, the first five entries of $P(t)$ represent the probability of a cell having no protein and a gene with $M_F$, $M_H$, $U_N$, $U_O$, and $O$ phases, respectively; the second five entries represent the probability of the cell having one protein and a gene in mentioned phases; and so on.
This probability vector evolves according to a continuous-time Markov process, which is called the chemical master equation (CME):
\begin{equation}\label{pdot}
\dot{P}(t)= A P(t) + D P(t),
\end{equation}
where the \emph{transition matrix} $A$ contains phase varying rates, and protein production and degradation rates, and $D$ is the  \emph{replication matrix}. According to the system's deterministic dynamics, we compute the building blocks of the transition matrix, i.e.,   phase variation matrix $K$ and protein production matrix $B$:
\begin{gather*}
K = \ml 0  & k_M & 0 & 0 &0\\
              0 &  -k_M & k_H & 0 & 0 \\
              0 & 0 &  - k_H -k_O & k_{-O} &0 \\
              0 & 0 & k_O &   - k_{-O} - k_R & k_{-R}\\
              0 & 0 & 0 & k_R &  - k_{-R} \mr, \\
B = \ml \beta_{\text{on}} & 0 & 0 &0 & 0 \\
          0 & \beta_{\text{on}}  & 0 & 0 &0 \\
          0 & 0 & \beta_{\text{partial}}  & 0 & 0 \\
           0 & 0 & 0 &  \beta_{\text{partial}} & 0\\
           0 & 0 & 0 &0 &  \beta_{\text{off}} \mr.
\end{gather*}
If we denote the identity matrix of size five by $I_5$, then
\begin{equation}\label{transitionmatrix}
A = \ml K-B & \gamma  I_5 & 0 &  \dots\\
            B & K-B - \gamma  I_5 & 2\gamma  I_5 & 0\\
            0 & B & K-B - 2\gamma  I_5 & 3\gamma I _5\\
           \vdots & &  \ddots & \\
 \mr.
\end{equation}
After replication, as mentioned in Section~\ref{replication}, any configuration transforms into another configuration with half protein concentration. Hence, the replication matrix can be written as $D = - I + D^+$. The negative identity matrix represents a continuous reduction in the probability of all configurations due to reduction in protein concentration. 
 The $D^+$ matrix contains the information on phase change and is composed of the blocks 
\begin{equation}\label{Dplus}
 D^+_{i,j}=\ml 0 &0&0&0&0 \\
                 1&0.5& 0 &0&0\\
                 0&0.5 &1 &0&0\\
                 0&0& 0 &1&0\\
                 0&0&0&0&1
\mr,\end{equation}
where the protein concentration of the $i$th five configurations is approximately half of that of the $j$th five configurations. 
Note that in computing the rates, one unit time is equal to one generation or the time between two replications.


\subsection{Fluorescence Grid Based Aggregation}\label{FGBAsec}


Aggregation or lumping of Markov chains has been known for a long time \cite{JGK-JLS:76}. 
Here, we aggregate the states of Markov chain $\dot{P}(t) = A P(t)$  into groups of $m_1, m_2, \dots$ states by a linear \emph{aggregation operator} $E$:
\begin{equation*}
 E = \ml \overbrace{1  \dots 1}^{m_1} & 0\dots  & & \\
                        0 \dots 0 & \overbrace{1 \dots 1}^{m_2} & 0  & \dots \\
&&&&\\
\vdots && \ddots &&
\mr.
\end{equation*}
Therefore, the aggregated probability vector at time $t$, is equal to
$P_{agg}(t) = E P(t).$
Taking the derivative of both sides gives 
$\dot{P}_{agg}(t) = E \dot{P}(t) = E A P(t).$
To find the dynamics of $P_{agg}(t)$ independent of $P(t)$, we define $P(t)$ as an approximate function of $P_{agg}(t)$. 
We assume that the probability of being in state $i$ is equal to the aggregated probability of being in the group that contains $i$ divided by the number of states in that group, that is,
\begin{equation}\label{Pagg}
P(t) \simeq F P_{agg}(t),
\end{equation}
where $F$ is the \emph{disaggregation operator} and 
\begin{equation*}
F = \ml   
\left.\begin{array}{c}
  \frac{1}{m_1} \\
  \vdots\\
  \frac{1}{m_1}\\ 
 \end{array} \right\} m_1  & \begin{array}{c} 0\\ \vdots \\ 0 \end{array}  & & \dots   \\
   \begin{array}{c} 0\\ \vdots \\ 0 \end{array}  & 
\left.\begin{array}{c}
       \frac{1}{m_2} \\
       \vdots\\
       \frac{1}{m_2}\\ 
       \end{array} \right\} m_2  &  \begin{array}{c} 0\\ \vdots \\ 0 \end{array}   & \dots  \\
  \vdots  & & \ddots &
\mr.
\end{equation*}
Now, consider the following \emph{approximated aggregated Markov chain} 
\begin{eqnarray}\label{Pa}
&& P_a(0) = P_{agg}(0) = E P(0), \nonumber\\
&& \dot{P}_a(t) = EAF P_a(t). 
\end{eqnarray}
Based on assumption~\eqref{Pagg}, the evolution of the solution $P_a(t)$ can approximate the evolution of the aggregated probability vector $P_{agg}(t)$.

\begin{remark}[Properties of aggregation]
First, our linear aggregation method is not \emph{lumpable} \cite{JGK-JLS:76}, or \emph{unbiased} \cite{ECH-CRJ:89}: a Markov chain is lumpable with respect to an aggregation if the transitions and states inside any partition group also compose a Markov chain. The class of Markov chains which admits this exact aggregation was investigated in \cite{JGK-JLS:76} and proved to be quite narrow. 
In our aggregation method the necessary and sufficient condition for lumpability of the CME is $CBAC = AC$. It is easy to see that this equality does not always hold, and thus our aggregation is not lumpable. Second, our aggregation is \emph{regular} \cite{ECH-CRJ:89}: an aggregation is regular if it is both linear and \emph{state partitioning}. Being state partitioning means that the aggregator should assign each state of Markov process to be aggregated to exactly one super state. It can be easily checked that our aggregation is state partitioning.
\end{remark}

Before proceeding to the theorem, define fluorescence rate matrix $B_f$ to be equal to a protein production matrix $B$ whose entries (e.g., $\beta_{\text{on}}$) are replaced by the corresponding fluorescence based production rates (e.g., $\beta_{f, \text{on}}$), see Section~\ref{transcription}.

\begin{theorem}[FGBA algorithm]\label{FGBA}
Consider a gene-protein system that can be described by the CME $\dot{P}(t) = A P(t)$, where the transition matrix $A$ is given by equation~\eqref{transitionmatrix}. From the experimental fluorescence histograms, extract the fluorescence grids $\Delta_1, \Delta_2, \dots$ and the fluorescence rate matrix $B_f$. Then the solution to the following fluorescence based CME simulates the experimental fluorescence histogram:
\begin{equation}\label{fCME}
\begin{array}{l}
\dot{P}_f(t)= A_f P_f(t), \\
A_f = \\
\ml K -\frac{1}{\Delta_1}B_f &  0  &  &  & \dots\\
            \frac{1}{\Delta_1}B_f    & K -\frac{1}{\Delta_2}B_f   & 0 &  & \\
            0                     & \frac{1}{\Delta_2}B_f   & K-\frac{1}{\Delta_3}B_f   & 0& \\
           \vdots & &   & \ddots \\
 \mr\\
+ \ml 0 &  \frac{\gamma \Delta_1}{\Delta_2} I_5  & 0 &  \dots \\
            0   &  - \frac{\gamma \Delta_1}{\Delta_2} I_5  & \frac{\gamma(\Delta_1+\Delta_2)}{\Delta_3} I_5   & 0\\
            0     & 0 &  - \frac{\gamma(\Delta_1+\Delta_2)}{\Delta_3} I_5  & \frac{\gamma(\Delta_1+\Delta_2+\Delta_3)}{\Delta_4} I_5 & 0 \\
           \vdots & &   & \ddots &\\
 \mr,
\end{array}
\end{equation}
where $P_f(0)$ is computed based on the initial state of the system in the experiments.
\end{theorem}
\begin{remark}
In Theorem~\ref{FGBA}, the phase variation matrix $K$ can be any arbitrary matrix with zero column sum. 
\end{remark}
\begin{proof}[Proof of Theorem~\ref{FGBA}]
First, we aggregate the states of the initial CME $\dot{P}(t) = A P(t)$ by lumping the configurations with different protein numbers but same phase, that is, the configurations along the x-axis of Figure~\ref{states}. 
Therefore, in above mentioned aggregation (disaggregation) operator, each entry $E_{ij}$ ($F_{ij}$) is replaced by a five by five block $E_{ij} I_5$ ($F_{ij} I_5$), and we denote the new aggregation (disaggregation) operator by $\overline{E}$ ($\overline{F}$).
Employing these operators,  the dynamics of the approximated aggregated CME will be:
\begin{equation}\label{approximatedCME}
\begin{array}{l}
P_a(0) = \overline{E} P(0)\\
\dot{P}_a(t)= A_a P_a(t), \\
A_a = \overline{E} A \overline{F} \\
=   \ml K -\frac{1}{m_1}B &  0  &  &  & \dots\\
            \frac{1}{m_1}B    & K -\frac{1}{m_2}B   & 0 &  & \\
            0                     & \frac{1}{m_2}B   & K-\frac{1}{m_3}B   & 0& \\
           \vdots & &   & \ddots \\
 \mr\\
+ \ml 0 &  \frac{\gamma m_1}{m_2} I_5  & 0 &  \dots \\
            0   &  - \frac{\gamma m_1}{m_2} I_5  & \frac{\gamma(m_1+m_2)}{m_3} I_5   & 0\\
            0     & 0 &  - \frac{\gamma(m_1+m_2)}{m_3} I_5  & \frac{\gamma(m_1+m_2+m_3)}{m_4} I_5 & 0 \\
           \vdots & &   & \ddots &\\
 \mr.\\
\end{array}
\end{equation}
In essence, the $(5i)$th entry of $P_a(t)$ represents the probability of having a protein concentration between $m_1+\dots+m_{i-1}$ and $m_1+\dots + m_i$ proteins at time $t$. Notice that the \emph{group sizes} $m_i$ and the protein production rates in $B$ are unknown.
Now, assume that for  $i =1, 2, \dots$, the each group size $m_i$ satisfies
\begin{equation}\label{deltam}
m_i = \max\{x_g \in \mathbb{N} | \hspace{.1in} \mu x_g \le \Delta_i\},
\end{equation}
where $x_g$ is the protein number and $\mu$ is the fluorescence-GFP ratio, defined in Section~\ref{transcription}. Roughly speaking, $m_i$ is the number of proteins in one cell that increases the fluorescent intensity of the cell by $\Delta_i$. Since for the experimental fluorescence grids in histograms of \cite{HL-AO:07}, $m_i$'s tend to be large, one can see that $\mu m_i \simeq \Delta_i$.
 Then, according to equation~\eqref{linearf},
$$\frac{\beta_*}{m_i} \simeq \frac{\beta_{f, *}/\mu}{\Delta_i /\mu} = \frac{\beta_{f, *}}{\Delta_i}.$$
Moreover, for any $i, j, k \in \{1, 2, \dots \}$, $$\frac{m_i+m_j}{m_k} \simeq \frac{(\Delta_i+\Delta_j)/\mu}{\Delta_k/\mu} = \frac{\Delta_i+\Delta_j}{\Delta_k}$$
Therefore, the fluorescence based CME~\eqref{fCME} is a direct consequence of approximated aggregated CME~\eqref{approximatedCME} under assumption $\mu m_i =\Delta_i$, and the unknown values $\mu$ and $m_i$'s are eliminated.
Note that the $i$th entry of $P_f(t)$ is now the probability of cell having fluorescent intensity between $\Delta_1+\dots+\Delta_{i-1}$ and $\Delta_1+\dots+\Delta_i$. 
\end{proof}

\begin{proposition}[Evolution of error in FGBA method]
Consider the dynamics of a gene-protein system with only one gene phase, hence one protein production rate $\beta$, is described by the CME $dot{P}(t) = A P(t)$. By employing the FGBA method, the system's dynamics can be approximated by the fluorescence based CME $\dot{P}_f(t)= A_f P_f(t)$. Assume that: 
\begin{enumerate}
\item\label{cond1} there exists $r \in \mathbb{R}_{> 0}$ such that the fluorescence grids satisfy 
$\Delta_i \le r \Delta_{i-1}$; and
\item\label{cond2} there exists $\epsilon \in \mathbb{R}_{> 0}$ such that the group sizes satisfy $|\mu m_i - \Delta_i| \le \epsilon$.
\end{enumerate}
Let $e(t)$ denote the error in the expected value of the final probability distribution, that is,
$$e(t) = \mu E[P(t)] - E[P_f(t)],$$ 
then $e(t)$ can be upper bounded by a well defined function of $E[P_f(t)]$, $\epsilon$, $r$, and the minimum and maximum fluorescence grid.
\end{proposition}
\begin{proof}
This error in FGBA method is caused by two reasons: aggregating the states of initial CME; and approximating the group sizes $m_i$ by fluorescence grid sizes, instead of employing the exact equation~\eqref{deltam}. Therefore,
\begin{multline*}
e(t) = \mu E[P(t)] - E[P_f(t)] \\
= (\mu E[P(t)] - \mu E[P_a(t)]) + (\mu E[P_a(t)] - E[P_f(t)]) \\
= \mu e_1(t) + e_2(t).
\end{multline*}
We first compute the first term's upper bound:
\begin{multline*}
e_1(t) = E[P(t)] - E[P_a(t)] \\
= [0 \;\; 1\;\;2 \; \dots ] P(t) - [0 \;\;  m_1 \;\; m_1+m_2 \; \dots ] P_a(t).
\end{multline*}
Taking the derivative of both sides gives
\begin{multline*}
\dot{e}_1(t) = [ 0 \;\; 1 \; \dots ] A P(t) \\ -  [0 \;\;  m_1 \;\; m_1+m_2 \;\; \dots ] A_{a} P_{a}(t),
\end{multline*}
where $A_a = E A F$, and $E$ and $F$ are the aggregation and disaggregation operators introduced in Section~\ref{FGBAsec}.
Hence,
\begin{multline*}
 \dot{e}_1(t) = [\beta \;\; \beta-\gamma \;\; \beta-2\gamma \;\;  \dots ]P(t) \\- [\beta \;\; \beta- \frac{m_1^2}{m_2}\gamma \;\;  \beta- \frac{m_2(m_1+m_2)}{m_3}\gamma \;\; \dots ]P_{a}(t) \\ 
= \beta {\bf 1}^T  P(t) - \gamma [ 0 \;\; 1 \;\; 2 \; \dots] P(t) \\
 - \beta {\bf 1}^T  P_{a}(t)  + \gamma [0 \;\; \frac{m_1^2}{m_2} \;\; \frac{m_2(m_1+m_2)}{m_3} \dots] P_{a}(t).  
\end{multline*}
Clearly, ${\bf 1}^T  P_{a}(t) = {\bf 1}^T  P(t) =1$. By adding and subtracting $\gamma e_1(t)$ we have
\begin{multline*}
\dot{e}_1(t) =  \gamma e_1(t)\\ + \gamma [0 \;\; \frac{m_1^2}{m_2}-m_1 \;\; \frac{m_2(m_1+m_2)}{m_3} - (m_1+m_2) \; \dots] P_{a}(t).  
\end{multline*}
Integrating from 0 to $t$ gives
\begin{multline*}
e_1(t) = e^{\gamma t} e_1(0) - \int_0^t e^{\gamma(t-\tau)} \gamma \\ [ 0 \;\; \frac{m_1(m_2-m_1)}{m_2} \;\; \frac{(m_1+m_2)(m_3-m_2)}{m_3}  \; \dots \; ] P_{a}(\tau) d\tau. 
\end{multline*}
According to the initial value $P_a(0) = E P(0)$ we have
\begin{eqnarray*}
e_1(0)  &=& [0 \;\; 1\;\;2 \; \dots ] P(0) - [0 \;\;  m_1 \;\; m_1+m_2 \; \dots ] E P(0) \\
 &=& P_2(0) + \dots + (m_1-1) P_{m_1}(0) \\
&&+ P_{m_1+1}(0) + \dots + (m_2-1) P_{m_2}(0)\\
 &&+ \dots.
\end{eqnarray*}
Clearly, $e_1(0)$ is a convex combination of $\{ 1, 2, \dots, m_{max} \}$, where $m_{max}$ is the maximum group size. Hence, denoting the maximum fluorescence grid by $\Delta_{max}$,
$$e_1(0) \le m_{max} \le (\Delta_{max} +\epsilon)/\mu.$$
Therefore, 
\begin{multline*}
e_1(t) \le  \frac{\Delta_{max} +\epsilon}{\mu} e^{\gamma t}  -  \max_i\{\frac{m_i - m_{i-1}}{m_i}\} \\\int_0^t e^{\gamma(t-\tau)} \gamma 
[0 \;\;  m_1 \;\; m_1+m_2 \; \dots \;] P_{a}(\tau) d\tau. 
\end{multline*}
According to the assumptions~\ref{cond1} and \ref{cond2}, the value $\max_i\{\frac{m_i - m_{i-1}}{m_i}\} $ can be upper bounded by $1 - \Delta_{min} /(r \Delta_{min} + r\epsilon + \epsilon)$, which we denote by $\hat{r}$:
\begin{equation*}
e_1(t) \le  \frac{\Delta_{max} +\epsilon}{\mu} e^{\gamma t}  +  \hat{r} (1- e^{-\gamma t}) \gamma E[P_a(t)].
\end{equation*}
Second, we compute the upper bound on error $e_2(t)$:
\begin{multline*}
e_2(t) = \mu [0 \;\;  m_1 \;\; m_1+m_2 \;\; \dots \;]P_{a}(t) \\- [0\;\; \Delta_1 \;\; \Delta_1+\Delta_2 \;\; \dots\;]P_f(t)\\
\le [0\;\; \Delta_1 \;\; \Delta_1+\Delta_2 \;\; \dots\;]P_{a}(t) + \epsilon{\bf 1}P_{a}(t) \\
 - [0\;\; \Delta_1 \;\; \Delta_1+\Delta_2 \;\; \dots\;]P_f(t).
\end{multline*}
The value $ \epsilon{\bf 1}P_{a}(t) $ is equal to $\epsilon$, and using assumption~\ref{cond2}, one can compute the scalar function $g(\epsilon,t)$ such that $ e^{A_a t} \le g(\epsilon,t) e^{A_f t}$, then
\begin{equation*}
P_{a}(t) = e^{A_a t}P_a(0) \le g(\epsilon,t) e^{A_at} P_f(0) = g(\epsilon,t) P_f(t).
\end{equation*}
Consequently, 
\begin{equation*}
e_2(t) \le (g(\epsilon,t) - 1) E[P_f(t)] + \epsilon,
\end{equation*}
and finally
\begin{multline*}
\mu e_1(t) + e_2(t)\le  (\Delta_{max} +\epsilon)e^{\gamma t}  \\
+ \Big( \hat{r} (1- e^{-\gamma t}) \gamma +g(\epsilon,t) + \epsilon\Big)  E[P_f(t)].
\end{multline*}
\end{proof}

\section{Numerical Results}\label{numericalresults}

In the experiments done by Lim et al., they let six separate colonies of E. coli grow for 20 hours. Each colony started from a cell that contains a \emph{mutant} of agn43 with Off expression state. The gene was mutated by deleting different parts of the upstream sequences of agn43. They claimed that the only difference in the dynamics of gene-protein system in these mutants is the ratio $k_R/k_{-R}$, see Figure~\ref{transcriptionfig}. According to the steady state of phase varying dynamics, the ratio $k_R/k_{-R}$ is equal to the fraction of unmethylated cells with Off expression, and is experimentally found to be $15.8, 8.9, 5.5, 4.3, 1,$ and $0.1$ for the six mutants. 
Finally, they measured the fluorescence of the cells in each colony with flow cytometer and drew fluorescence histograms, see Figure~\ref{niz}.(a). In these histograms, the fluorescence grids $\Delta_i$ are equal to $10^i - 10^{i-1}$, where $i\in 0.05\{0, 2, \dots, 40\}$.

Now, to generate analytical fluorescence histograms, we employ the FGBA method stated in Theorem~\ref{FGBA} to the gene-protein systems of the mutants of agn43 (each system has one of the six mentioned values for $k_R/k_{-R}$, and the rest of parameters remains constant.) Knowing the phase variation rates, Section~\ref{phaserate}; degradation and fluorescence increase rates, Section~\ref{transcription}; and fluorescence grids, the fluorescence based transition matrix $A_f$ of equation~\eqref{fCME} can be computed. Therefore,
\begin{equation}\label{finalCME}
\dot{P}_f(t) = A_f P_f(t) + D_f P_f(t),
\end{equation}
where $D_f = -I + D^+_f$ is the fluorescence based replication matrix. The $ D^+_f$ is composed of blocks $D^+_{i,j}$, given by equation~\eqref{Dplus}, while the fluorescence of the $i$th five configurations is approximately half of the $j$th five configurations.
\begin{remark}
In Lim's model, the replication of the gene is described by a discrete time reaction. Accordingly, we first employed a discrete time replication in our stochastic analysis. However, the variance of the resulting probability distribution did not match the variance of the experimental fluorescence histograms, see Figure~\ref{discrete}. This disagreement can be explain as follows: The experimental histograms are taken from a colony of the cells, and in a colony not all the cells replicate at the same time. Therefore, a continuous time replication can better capture the behavior of a large number of cells than a discrete time replication. 
\end{remark}

Equation~\eqref{finalCME} is an infinite dimensional ODE, hence we truncate this equation into a finite dimensional equation. 
The finite dimensional CME should contain configurations whose protein concentration is between zero and the maximum number of proteins in one cell, or equivalently, configurations whose fluorescence is less than the maximum value observed ($10^4$ a.u.). The solutions to the final CME's for the six mentioned mutants are plotted in Figure~\ref{niz}.(b).



\begin{figure}
  \centering
    \includegraphics[width=3.5in,keepaspectratio]{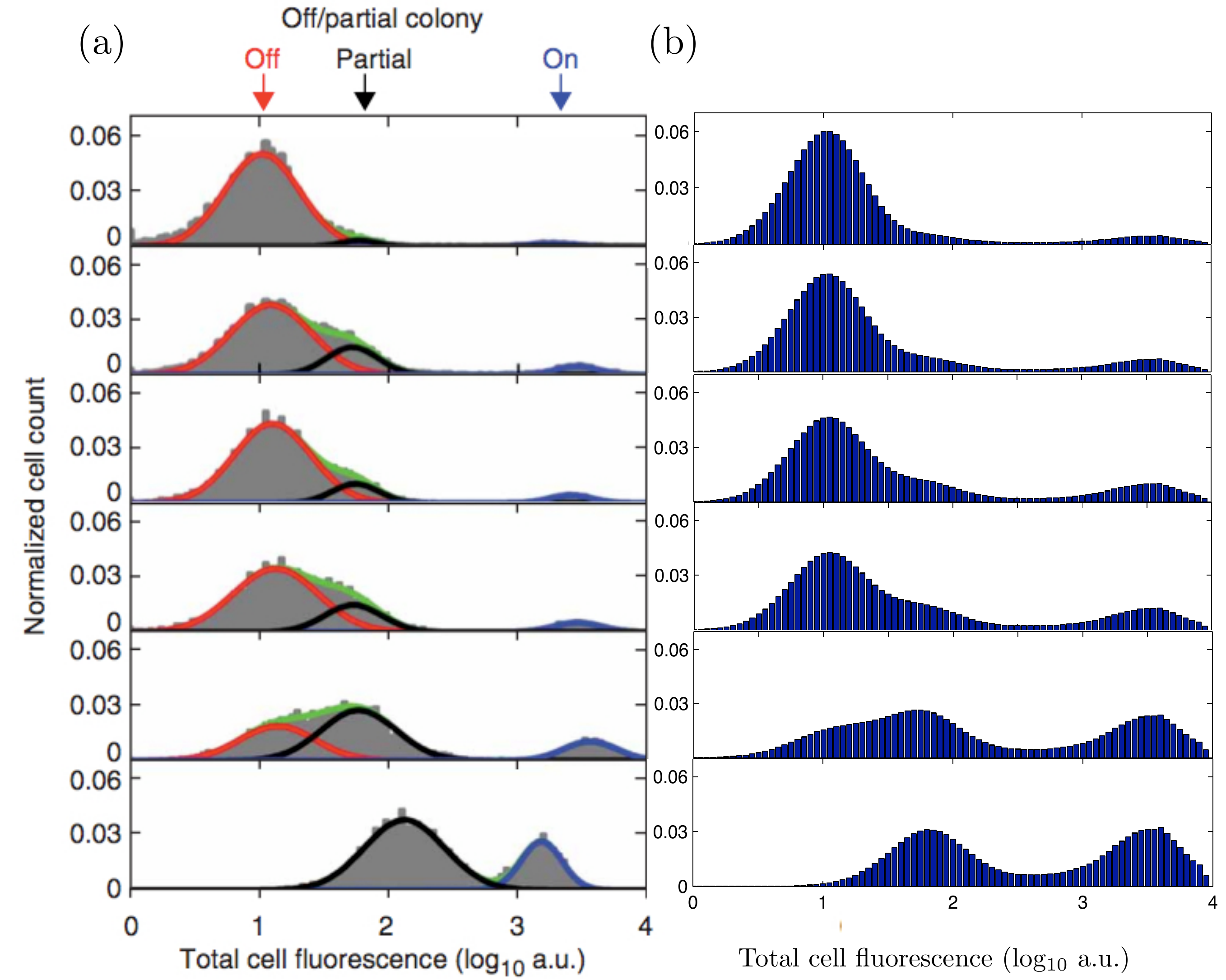}
  \caption{(a) The fluorescence histograms for six separate colonies, each starting from a mutant of agn43 with Off expression state. Lim et al. claims that the only difference in the dynamics of these mutants is the ratio $\frac{ k_R}{k_{-R}}$, which is equal to  
$15.8, 8.9, 5.5, 4.3, 1,$ and $0.1$ from top to bottom. Reprinted figure with permission from \cite{HL-AO:07}. \copyright 2007, by Nature Publishing Group. (b) The probability distribution of fluorescent intensity resulting from solving the aggregated CME of equation~\eqref{pdot}. Each plot is obtained by solving the model with one of the six mentioned values for $\frac{ k_R}{k_{-R}}$, and the rest of parameters remains constant. These plots proves that our method can make a phase varying model verifiable by the fluorescence histograms.}  \label{niz}
\end{figure}
\begin{figure}
  \centering
    \includegraphics[width=2.5in,keepaspectratio]{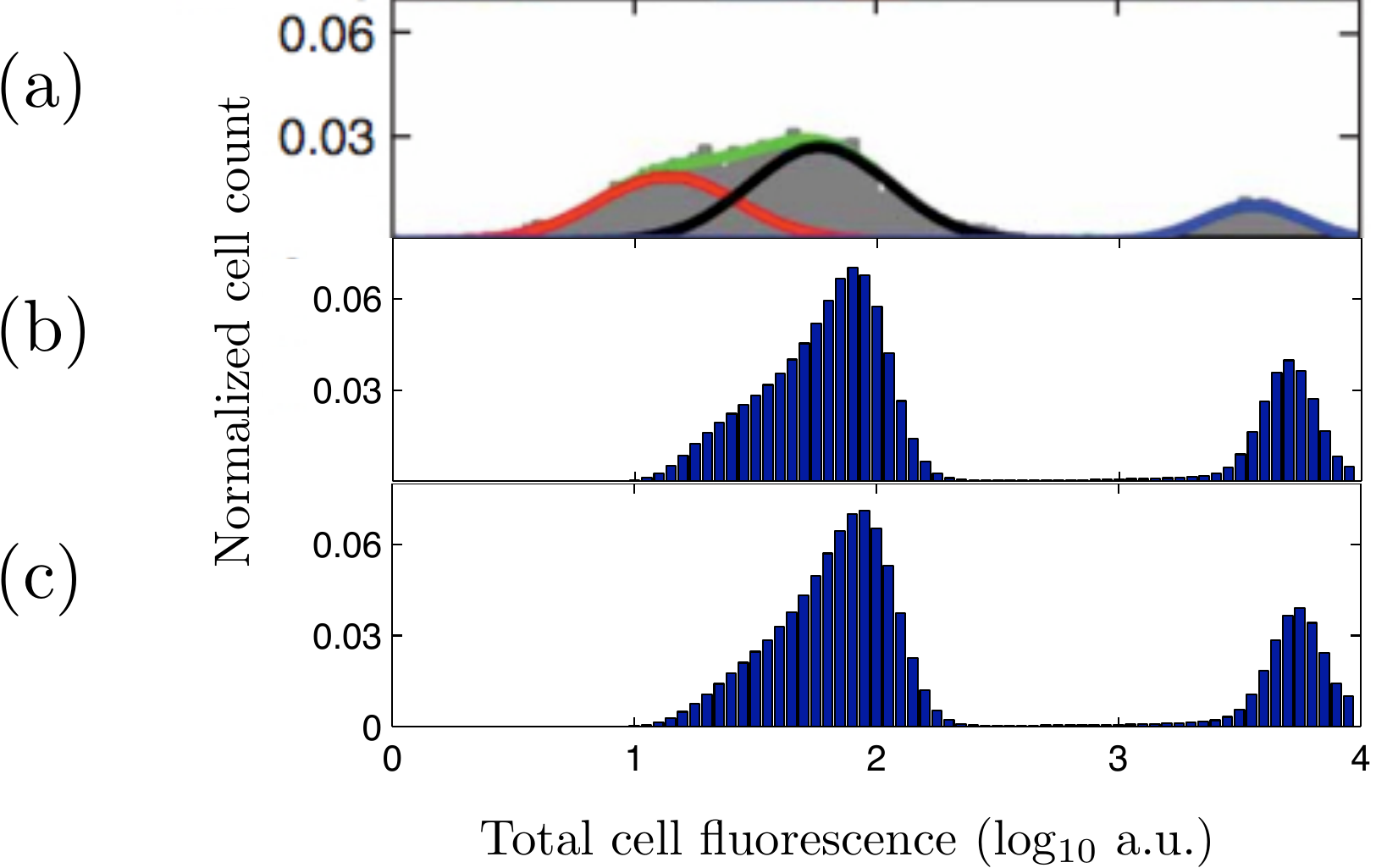}
  \caption{(a) The experimental fluorescence histogram of a colony that starts from a mutant of agn43 with Off expression state and the ratio $\frac{ k_R}{k_{-R}} = 1$. Reprinted figure with permission from \cite{HL-AO:07}. \copyright 2007, by Nature Publishing Group. 
(b)  The probability distribution of fluorescent intensity resulting from solving the fluorescence based CME with discrete time replications for the same mutant. We assumed that the fluorescence becomes half in each replication.
(c)   The probability distribution of fluorescent intensity resulting from solving the fluorescence based CME with discrete time replications for the same mutant. Here, the replication matrix is constructed by a binomial probability distribution, in order to increase the resulting variance. This figure tells us that a model with discrete time replication can not capture the variance of the experimental  fluorescent intensity distributions. Moreover, employing a binomial probability distribution for replication only slightly increases this variance.}
  \label{discrete}
\end{figure}

\section{Conclusion and Future Work}\label{conclusion}

 As our main result, we introduced a new approach to justify the dynamical model of protein expression by the experimental fluorescence histograms. 
We described the dynamics of a gene-protein system, whose protein production rates are unknown, with a chemical master equation (CME).
Based on the resolution of the experimental histograms, we aggregated  the states of the CME, however, the number of states in each aggregated group is also unknown.
We proved that the unknown protein production rates and number of states in one group can be replaced by the fluorescence increase rate and the fluorescence grids from the histograms, respectively.
Therefore, the final probability distribution is the theoretical fluorescent histogram of the gene-protein model, and can be verified by the experimental fluorescence histograms.
One future challenge is to compute the parameters of a gene-protein system via its fluorescence histograms. 
The solution to the CME, which is a probability distribution, has been numerically approximated from the parameters of the CME,  see \cite{JZ-TL-AC-CY:10}. A reverse analysis of this method can help us find the parameters of a gene-protein system by experimental fluorescence histograms.

\section{Acknowledgments}
 The authors would like to thank Marjan Van der Woude, Jo{\~a}o Hespanha, Brian Munsky, and Sandra Dandach  for 
their helpful comments and encouragements. 

\bibliographystyle{plain}
\bibliography{../bioref,../../ref/alias,../../ref/FB,../../ref/Main,../../ref/New}

\end{document}